\documentclass{article}
\usepackage{mathrsfs}

\usepackage{amsmath,amsfonts,amsthm,amssymb,amscd,color,xcolor,mathrsfs,verbatim,microtype}
\usepackage{graphicx,eurosym}
\usepackage{amssymb,amsmath,comment}
\usepackage{amsfonts,mathrsfs}
\usepackage{color}
\usepackage{bbm}
\usepackage{bm}
\usepackage{hyperref}
\usepackage{stmaryrd} 

\usepackage[applemac]{inputenc}

\usepackage[cyr]{aeguill}

\colorlet{darkblue}{blue!50!black}

\hypersetup{
    colorlinks,%
    citecolor=blue,%
    filecolor=red,%
    linkcolor=darkblue,%
    urlcolor=blue,%
    pdfnewwindow=true,%
    pdfstartview={FitH}
}

\usepackage{graphicx,amscd,mathrsfs,wrapfig,mathrsfs,lipsum}
\usepackage{eufrak}
\usepackage{float}
\usepackage{tikz}
\usepackage{multicol}
\usepackage{caption}
\usetikzlibrary{arrows}
\usepackage{capt-of}

\colorlet{darkblue}{blue!50!black}

\usepackage{marginnote}

\newcommand{\R}{{\mathbb R}}

\newcommand{\ty}{\infty}

\newcommand{\KK}{{\cal K}}

\newcommand{\PP}{{\cal P}}

\def\mE{{\mathbb E}}
\def\dif{{\mathord{{\rm d}}}}
\def\cF{{\mathcal F}}

\def\mP{{\mathbb P}}

\def\mR{{\mathbb R}}

\def\cA{{\mathcal A}}

\def\cV{{\mathcal V}}
\def\cR{{\mathcal R}}

\def\eps{\varepsilon}
\def\cL{{\mathcal L}}

\newcommand{\lag}{\langle}
\newcommand{\rag}{\rangle}

\newcommand{\dd}{{\textup d}}

\newcommand{\BBBBB}{{\mathcal B}}

\theoremstyle{plain}

\newtheorem*{lemma*}{Lemma}
\newtheorem{theorem}{Theorem}[section]
\newtheorem{lemma}[theorem]{Lemma}
\newtheorem{proposition}[theorem]{Proposition}

\theoremstyle{definition}
\newtheorem{definition}[theorem]{Definition}

\theoremstyle{remark}
 \newtheorem{remark}[theorem]{Remark}

\numberwithin{equation}{section}

\begin{document}

\author{
Xuhui
Peng\,\footnote{1.\,MOE-LCSM, School of Mathematics and Statistics, Hunan Normal University, Changsha, Hunan410081, P.R.China;
 2.\,Key Laboratory of Control and Optimization of Complex Systems, Hunan Normal University, College of Hunan Province, Changsha 410081, China
 e-mail: \href{mailto:xhpeng@hunnu.edu.cn}{xhpeng@hunnu.edu.cn}}\,
\and
Lihu  Xu\,\footnote{1.\,Department of Mathematics, Faculty of Science and Technology, University of Macau, Macau, China; 2.\,Zhuhai UM Science \& Technology Research Institute, Zhuhai, China.  e-mail: \href{mailto:lihuxu@umac.mo}{lihuxu@um.edu.mo}} }

 \date{\today}

\title{Large deviations principle  for  2D Navier-Stokes equations with  space time localised noise}
\date{\today}
\maketitle

\begin{abstract}
We consider a stochastic 2D Navier-Stokes equation in a bounded domain.
The random force is assumed to be non-degenerate and periodic in time, its law has a support localised with respect to both  time and space. Slightly strengthening the conditions in the pioneering work about exponential ergodicity by Shirikyan \cite{shirikyan-asens2015},
we prove that the stochastic system satisfies Donsker-Varadhan typle large deviations principle. Our proof is based on a criterion of \cite{JNPS-cpam2015} in which we need to verify uniform irreducibility and uniform Feller property for the related Feynman-Kac semigroup.



\end{abstract}


\setcounter{section}{0}

\section{Introduction}
\label{s0}
In this paper, we consider a 2D Navier-Stokes system in a bounded  domain
$D\subseteq \mR^2$
with smooth boundary $\partial D:$
\begin{eqnarray}
\label{3-1}
\left\{
\begin{split}
  & \partial_t{u}+\langle u,\nabla u\rangle-\nu \Delta u+\nabla p=\eta(t,x), \text{div}
  u=0,x\in D
  \\  & u|_{\partial D}=0,
  \\   & u(0,x)=u_0(x).
  \end{split}
  \right.
\end{eqnarray}
Here $u=(u_1,u_2)$  and $p$ are the velocity and the pressure of the fluid respectively, $\nu>0$ is the viscosity, and $\eta$ is the external random force to be specified below.

Let $I_k$  be  the indicator function of the interval $(k-1,k)$ and assume that   $\eta_k$
is a sequence of i.i.d random variables in $L^2([0,1]\times  D)$ which  is localised in both time and space; see   the  hypothesis \textbf{(H1)} below for details.
In this paper,
$\eta$ is a random force with the form
\begin{eqnarray}
  \eta(t,x)=\sum_{k=1}^\infty I_k(t)\eta_k(t-k+1,x), \quad t\geq 0.
\end{eqnarray}
%
Let us denote by $\textbf{n}$ the outward normal to the boundary $\partial D$ and introduce the space
\begin{eqnarray}
\label{5-1}
  H=\big\{u\in L^2(D,\mR^2):\text{div } u=0 \text{ in } D, \langle u,\textbf{n} \rangle=0
  \text{ on } \partial D \big\}
\end{eqnarray}
which will be endowed with the usual $L^2$ norm $\|\cdot \|.$

We fix  an  open set $Q\subseteq [0,1] \times  D$ and denote by $\{\phi_j\}\subseteq H^1(Q,\mR^2)$
an orthonormal basis in $L^2(Q,\mR^2)$. Let $\chi\in C_0^\infty(Q)$ be a non-zero function and let $\psi_j=\chi \phi_j.$
The following hypotheses \textbf{(H1)} and \textbf{(H2)} are both from \cite{shirikyan-asens2015}.

\begin{description}
  \item \textbf{(H1)} \textbf{Structure of the noise.} The random variables $\eta_k$ have the following form
\begin{eqnarray}
  \eta_k=\sum_{j=1}^\infty b_j\xi_{jk} \psi_j(t,x),
\end{eqnarray}
where $\xi_{jk}$ are { i.i.d.} scalar random variables such that
$|\xi_{jk}|\leq 1$ with probability 1, and  $\{b_j\}\subseteq \mR$ is a non-negative sequence such that
\begin{eqnarray*}
  B:=\sum_{j=1}^\infty b_j\|\psi_j\|_1<\infty,
\end{eqnarray*}
where $\|\cdot \|_1$  denotes  usual Sobolev norm  on  $H^1(Q,\mR^2)$.
Moreover,  $\xi_{jk}$  has a $C^1$-smooth density $\rho_j$ with respect to the Lebesgue measure on the real line.
\end{description}

Let us denote by $\mathcal{K} \subseteq L^2(Q,\mR^2)$ the support of the law of $\eta_k.$ The hypotheses imposed on $\KK$ imply that $\KK$ is a compact subset in $H_0^1(Q,\mR^2).$

\begin{description}
\item \textbf{(H2)} \textbf{Approximate controllability.}
There exists a $\hat{u}\in H$ such that for any positive constants $R$ and $\eps$ one can find an integer $\ell \geq 1$ with the following property: given $v\in B_H(R):=\{u\in H, \|u\|\leq R\},$
 there are { $\theta_1,\cdots,\theta_\ell \in \mathcal{K} $} such that
 \begin{eqnarray}
   \|S_\ell(v,\theta_1,\cdots,\theta_\ell)-\hat{u}\|\leq \eps,
 \end{eqnarray}
 where $S_\ell(v,\theta_1,\cdots,\theta_\ell)$ is the vector $u_\ell$ defined by
 (\ref{4-1}) with $\eta_k=\theta_k$ and $u_0=v.$
\end{description}
In order to prove the large deviations principle, we need the following additional condition:
\begin{description}
  \item \textbf{(H3)} \textbf{Nondegenerate noise.} $b_j \ne 0$ for all $j$ with $b_j$ being defined in \textbf{(H1)}.
\end{description}

{By \cite[Theorem 2.1.18]{KS-book},  (\ref{3-1}) admits a unique strong solution $u(t)$.}
Since $u(1)$ depends on $u_0$ and $\eta_1$, we use
$S(u_0,\eta_1) $ to denote it. With these notations, we have
\begin{eqnarray}
\label{4-1}
  u_k=S(u_{k-1},\eta_k), \quad  k\geq 1.
\end{eqnarray}
As $\eta_k$ are i.i.d random variables in $L^2([0,1]\times  D)$, (\ref{4-1}) defines a homogeneous  family of Markov chains in $H.$ We denote it by $(u_k,\mP_u), u\in H$ and   use $P_k(u,\Gamma)$ to denote the transition function for $(u_k,\mP_u).$

The objective of this paper is to prove
the large deviations of the following  occupation measures
under hypotheses   \textbf{(H1)}--\textbf{(H3)}
\begin{eqnarray*}
  \zeta_k =\frac{1}{k}\sum_{j=0}^{k-1}\delta_{u_j}, k\geq 1.
\end{eqnarray*}
See Theorem \ref{1-7} below for  details.

There have been many literatures on the ergodicity  of stochastic 2D Navier-Stokes system,  see \cite{Flandoli95, EMS-2001, HM-2006,HM-2008} for the Gaussian noises case and \cite{KS-cmp2000} for the kick noises case, the monograph \cite{KS-book} gives a comprehensive review on the researches in these two directions. When the driven noises are L\'evy type, we refer the reader to \cite{DX-2011}, \cite{DXZ-2011} and the references therein.

There are very few papers on the ergodicity of stochastic partial differential equations (SPDEs) driven by physically localized noises. To the best of our
knowledge, there exist only { four}
 papers in this direction. Besides \cite{shirikyan-asens2015}, Shirikyan \cite{shirikyan-cup2017} proved the exponentially mixing for one-dimensional Burgers equation perturbed by a
stochastic forcing which is white in time and localized in space.
 When the noise is localized in physical space and degenerate in Fourier space,  Nersesyan \cite{VN-2021}
 proved that  the complex Ginzburg-Landau equation is exponential mixing.
{For the 2D Navier-Stokes system driven by a random force acting through the boundary, Shirikyan   \cite{Shi21} established  an    exponential mixing  property.}

 However,
there seem no works concerning the Donsker-Varadhan type large deviations principle (LDP) for SPDEs driven by physically localized noise,
the main motivation of this paper is to partly fill in this gap. { Continuing the pioneering work in \cite{shirikyan-asens2015} and
 slightly strengthening the conditions therein, we prove that the system \eqref{3-1} satisfies Donsker-Varadhan typle LDP.}

Our proof is based on a criterion established in \cite{JNPS-cpam2015} by verifying uniform irreducibility and uniform Feller properties for the related Feynman-Kac semigroup. Note that Donsker-Varadhan type LDPs have been extensively studied since the work \cite{DV-1975-83}. When Markov processes are strong Feller and irreducible,  Wu \cite{Wu-2001} established the hyper-recurrence criterion which were applied to study several SPDEs, see \cite{gourcy-2007b, gourcy-2007a, WXX-2021}.

  \subsubsection*{Acknowledgements}
  We would like to gratefully thank Vahagn Nersesyan for his kindly suggesting this research topic to us.
    Xuhui Peng would like to gratefully thank University  of  Macau for the hospitality, his research is supported in part by
  by NNSFC (No. 12071123), Scientific Research
Project of Hunan Province Education Department (No. 20A329) and  Program of Constructing Key Discipline in Hunan Province. Lihu Xu is supported in part by NNSFC grant (No.12071499), Macao S.A.R grant FDCT 0090/2019/A2 and University of Macau grant  MYRG2020-00039-FST.

 \subsubsection*{Notation}

In this paper, we use the following notation

\smallskip
\noindent
$H$ is the Hilbert space defined by (\ref{5-1}),    $B_H(a,R)$   is the closed   ball in~$H$ of radius~$R$ centred at~$a$. When~$a=0$, we simply write it as $B_H(R)$.


\smallskip
\noindent
For an  open set  $Q$ of a Euclidean space,   $H^m= H^m(Q)$  is the    Sobolev space of order $m$. We endow the space  $H^m$ with the usual Sobolev norms which are denoted by $\|\cdot \|_m$ or  $\|\cdot \|_{H^m(Q)}$.
$H_0^m=H_0^m(Q)$ is the closure in $H^m$ of the space of infinitely smooth functions with compact support.

\smallskip
\noindent
If $X$ is a metric space,  $L^\infty(X)$ denotes  the space of bounded Borel-measurable functions $f:X\to\R$ endowed~with the norm $\|f\|_\infty=\sup_{u\in X}|f(u)|$.

\smallskip
\noindent
$\BBBBB(X)$ is the    Borel $\sigma$-algebra of $X$.
$\PP(X)$ denotes  the set of  probability measures on $X$
endowed~with the topology of  the  weak convergence.
 For $\mu\in \PP(X)$ and~$f\in L^\ty(X)$,  we denote~$\lag f,\mu\rag=\int_{X}f(u)\mu(\dd u).$


\smallskip
\noindent We denote by $\mathcal{E}_m$ the vector span of $\psi_1,\cdots,\psi_m$
endowed with the $L^2$ norm and by $B_R$ the ball in $H^1(Q)$ of radius $R$ centred at origin.
$P_m$ stands the orthogonal projection in $L^2(Q,\mR^2)$ onto the m-dimensional space
$\mathcal{E}_m$.

\smallskip
\noindent
Let $\{e_j\}_{j=1}^\infty$ be an orthonormal basis in $H$ which consists of the eigenfunctions of the Stokes operator $L=-\nu\Pi \Delta,$ where $\Pi$ stands for the Leray projection in $L^2(D,\mR^2)$ onto the closed space $H.$
Let $\Pi_N$ be the orthogonal projection in $H$ on the vector space $H_N$
spanned by $e_1,\cdots,e_N.$

 \section{Main results}\label{S:1}

Let $X$ be a polish space  and let $\PP(X)$ be the space of probability measures on $X$ endowed with the topology of weak convergence.
A random probability measure  on  $X$ is defined as a measurable   mapping from a probability space $(\Omega,\cF,\mP)$ to $\PP(X).$
{A mapping $I : \PP(X) \to [0, +\ty]$ is  
 called a  \emph{rate function} 
 if it is lower-semicontinuous, and a rate function $I$ is said to be \emph{good}
 if   its level set~$\{\sigma\in\PP(X) : I(\sigma) \le \alpha \}$ is compact for any~$\alpha \ge0$.}
 A good rate function~$I$ is nontrivial
if its effective domain $D_I = \{\sigma \in  \PP(X) : I(\sigma) < \ty\}$ is not a singleton.

\begin{definition}
Let $\{\zeta_n,n\geq 1\}$  be a  sequence of random probability measures on $X$.       We say that  $\{\zeta_n,n\geq 1\}$  satisfies an LDP   with a    rate function $I: \PP(X)\to [0,+\ty]$,  if  the following two bounds~hold:
\begin{description}
\item[Upper bound.]
For any closed subset~$F\subset\PP(X)$, we have
$$
\limsup_{k\to\infty} \frac1k\log \sup \mP\{\zeta_k\in F\}\le -\inf_{\sigma\in F} I(\sigma).
$$
\item[Lower bound.]
For any open subset~$G\subset\PP(X)$, we have
$$
\liminf_{k\to\infty} \frac1k\log \inf\mP\{\zeta_k\in G\}\ge -\inf_{\sigma\in G}I(\sigma).
$$
\end{description}
\end{definition}

Let  $u_0$ be an arbitrary random variable in $H$ which is  independent of $\{\eta_k\}$ and  define a family of occupation measures  $\{\zeta_k\}$ by
\begin{align}
\label{px-1}
  \zeta_k=\frac{1}{k}\sum_{j=0}^{k-1}\delta_{u_j},
\end{align}
where $\delta_u$ is the delta measure concentrated at $u.$

In order to state our main result, let us first give the definition of $\cA_k$ and $\cA$ as below. For a  closed subset $B\subset H,$ define the sequence of sets
\begin{eqnarray*}
  \cA_0(B)=B,\quad { \cA_k(B)=\big\{ S(u,\eta): u\in \cA_{k-1}(B), \eta\in \mathcal K \big\}}, ~ k\geq 1,
\end{eqnarray*}
and denote
$$\cA(B)=\overline{\bigcup_{k \ge 0} \cA_k(B)},$$
where the closure is in $H$.
We shall call  $\cA(B)$ the domain of attainability from $B.$
Let $\hat u \in H$ be defined in {\bf (H2)}, we denote
$$\cA_k=\cA_k(\{\hat{u}\}) \  \ \forall k \ge 0, \ \ \ \  \ \cA=\cA(\{\hat{u}\}).$$

{Since $\cA$ a close subset of the Hilbert space $(H,\|\cdot\|),$
$(\cA,\|\cdot\|)$ is a closed metric space.
For this space, we use  $C(\cA)$  to  denote   the space of  continuous functions on $\cA$,
and  use  $C^1(\cA)$  to   denote    the space of  functions  on $\cA$ that are continuously Fr\'echet differentiable.}

Shirikyan \cite{shirikyan-asens2015}   proved that the stochastic system (\ref{4-1})
  possesses a unique stationary distribution $\mu$, which is exponentially
mixing.
  In Proposition   \ref{7-1} below,
  we further show that prove that the support of measure $\mu$ is $\cA$, i.e.,
  $$supp(\mu)=\cA.$$
The following theorem is the main result of this paper, its proof will be given in Section \ref{S:3}.
\begin{theorem}
\label{1-7}
  Let Hypotheses \textbf{(H1)},\textbf{(H2)} and \textbf{(H3)} hold  and let $u_0$ be an arbitrary random variable in $H$ whose law is supported by $\cA$.
   Then the family  $\{\zeta_k,k\geq 1\}$ of random measures on $\cA$ satisfies the LDP with a good rate function $I$ defined by
   \begin{eqnarray}
     I(\sigma)=\sup_{V\in C(\cA)} \big(\langle V,\sigma\rangle-Q(V) \big),
     \quad \sigma \in \PP(\cA),
   \end{eqnarray}
   where $Q(V)$ is defined as
   \begin{eqnarray}
     Q(V)=\lim_{k\rightarrow \infty}\frac{1}{k}\log \mE \exp{(\sum_{j=1}^kV(u_j))},
   \end{eqnarray}
  this limits  exists for any $V\in C(\cA)$\footnote{In the beginning of Section \ref{j-7}, we will prove that  $\cA$ is a compact set. Therefore,  if  $V\in C(\cA)$, then $V$ is also bounded.}
   and does not depend on the initial point $u_0.$
\end{theorem}

\begin{remark}
The sequence  $\{\zeta_k\}$  also  satisfies an LDP on $H$ under the conditions of Theorem  \ref{1-7}.
  More precisely, there is a lower-semicontinuous  mapping $I:\PP(H)\rightarrow [0,+\ty]$
  such that
  \begin{align*}
  \begin{split}
    -\inf_{\lambda \in \dot{\Gamma}} I(\lambda) & \leq \liminf_{k\rightarrow \infty}
    \frac{1}{k}\log \mP(\zeta_k\in \Gamma)
    \\ &\leq
    \limsup_{k\rightarrow \infty} \frac{1}{k}\log \mP(\zeta_k\in \Gamma)
    \leq -\inf_{\lambda \in \overline{\Gamma}} I(\lambda),
    \end{split}
  \end{align*}
  where  $\Gamma \subset \PP(H)$
  is an arbitrary Borel subset and $\dot{\Gamma}$ and $\overline{\Gamma}$
  denote its interior and closure, respectively.
\end{remark}

\begin{remark}
Note that  \cite{shirikyan-asens2015} assumes that the random force has the form $f(t,x)=h(t,x)+\eta(t,x)$ with $h(t,x)$ being some periodic function in $H^1(Q,\mR^2)$. Because $h(t,x)$ will not play an essential role in proving large deviations, without loss of generality we assume $h(t,x) \equiv 0$ in this paper.
\end{remark}

\section{Proof of  Theorem \ref{1-7}} \label{S:3}
\label{j-7}
First, let us show  the compactness of $\cA$ in $H$ which   plays an  important role  in our proof.
By  \cite[Proposition 2.1.21]{KS-book}, one has
\begin{eqnarray*}
  \|S(u,\eta_k)\|\leq \kappa \|u\|+C_1 \|\eta_k\|_{L^2(Q)},
\end{eqnarray*}
where $\kappa<1,C_1>0$ are some universal constants.
Thanks to hypothesis \textbf{(H1)}, we can choose
$r>0$  big enough  so that
\begin{eqnarray*}
\|\eta_k\|_{L^2(Q)}\leq r
\end{eqnarray*}
 with probability 1.
Then, with probability 1, one has
\begin{eqnarray*}
  \|S(u,\eta_k)\|\leq \kappa \rho +C_1 r, ~\forall u \in B_H(\rho).
\end{eqnarray*}
It follows that if $R\geq \frac{C_1r}{1-\kappa},$
then the ball $B_H(R)$ is invariant for the Markov chain $(u_k,\mP_u).$
Let us take $R$ big enough such that
\begin{align*}
R>\frac{C_1r}{1-\kappa}
\end{align*}
and
$\hat{u}\in B_H(R).$
Let  $X$  be the image of the set $B_H(R)\times B_{L^2(Q)}(r)$ under the mapping $S,$
then, {for any $u\in X $ and $\eta_k\in B_{L^2(Q)}(r)$, $S(u,\eta_k)\in X.$ }
By the compact embedding theorem(see \cite[Section 5.7]{Eva10} for example) and
\cite[(2.52)]{KS-book},
$X$  also is  a compact set  in $H$.
{Therefore, we obtain the compactness of $\cA$  from the
  fact  that  $\cA \subseteq X$ is a close set.}

{
We endow  the set   $\cA$ with distance $d(u,v)=\|u-v\|,$  then $\cA$ is a compact metric space.
Theorem \ref{1-7} immediately follows from Proposition \ref{px-5}, as long as we verify uniform irreducibility and uniform Feller property which are given in Section \ref{subsection3:1} and Section \ref{subsection3:2} respectively.
Hence, we complete the  proof of Theorem  \ref{1-7}.
}

\begin{proposition}
\label{px-5}
  Let $\{u_k,k\geq 1\}$ be a homogeneous family of Markov chains on a compact metric space $(\cA,d(\cdot,\cdot))$  with transition  function $P_k(u,\cdot).$
 The family  $\{\zeta_k,k\geq 1\}$ of random measures on $\cA$  is defined by (\ref{px-1}).
 Assume   the following two conditions hold:

  \noindent  \textbf{Uniform Irreducibility.} For any $r>0,$
    there is an integer $m\geq 1$ and a constant $p>0$ such that
    \begin{eqnarray}
    \label{h-2}
      P_m(x,B(a,r))\geq p,\quad \forall x,a \in \cA,
    \end{eqnarray}
where $B(a,r)=\{y\in \cA: d(y,a)\leq r\}.$

 \noindent  \textbf{Uniform Feller property:}
 For any $V,f\in C^1(\cA)$, the sequence $\{\|\mathcal{B}_n^V\textbf{1}\|_{\infty}^{-1}\mathcal{B}_n^Vf\}$ is uniformly equicontinuous  on $\cA$\footnote{The  uniformly equicontinuous of $\{\|\mathcal{B}_n^V\textbf{1}\|_{\infty}^{-1}\mathcal{B}_n^Vf\}$  on $\cA$ means that for any $\eps>0$, there is a $\delta>0$ such that $\big| \|\mathcal{B}_n^V\textbf{1}\|_{\infty}^{-1}\mathcal{B}_n^Vf(u_1) -\|\mathcal{B}_n^V\textbf{1}\|_{\infty}^{-1}\mathcal{B}_n^Vf(u_2)\big|\leq \eps$
 for any $u_1,u_2\in\cA $ with  $d(u_1,u_2)\leq \delta$.}.
Here,
for any $u_0\in \cA $ and $V,f\in C(\cA), $
$\mathcal{B}_n^Vf(u_0)$ and $\|\mathcal{B}_n^V f\|_{\infty} $ are defined by
 $$
 \mathcal{B}_n^Vf(u_0)= \mE_{u_0}\big[\exp\{\sum_{i=1}^nV(u_i)\}f(u_n)\big]\quad \text{ and }\quad
 \|\mathcal{B}_n^V f\|_{\infty}= \sup_{u_0\in \cA }|\mathcal{B}_n^Vf(u_0)|,
$$
 respectively.
 In the above, the subscript $u$  in $\mE_u$ is the initial state of
  the Markov chain $\{u_k\}$. We call $\mathcal{B}_n^V$  the Feynman-Kac semigroup with respect to $V$.

  Then the family  $\{\zeta_k,k\geq 1\}$  satisfies the LDP with a good rate function $I$ defined by
   \begin{eqnarray}
   \label{px-2}
     I(\sigma)=\sup_{V\in C(\cA)} \big(\langle V,\sigma\rangle-Q(V) \big),
     \quad \sigma \in \PP(\cA),
   \end{eqnarray}
   where $Q(V)$ is defined as
   \begin{eqnarray*}
     Q(V)=\lim_{k\rightarrow \infty}\frac{1}{k}\log \mE \exp{(\sum_{j=1}^kV(u_j))},
   \end{eqnarray*}
  this limits  exists for any $V\in C(\cA)$ and does not depend on the initial point $u_0.$
\end{proposition}
\begin{proof}
The arguments below is essentially from \cite{JNPS-cpam2015}, we only sketch its main ingredients.

 By the abstract result established by  \cite{kifer-1990}(see also \cite[Theorem A.1]{JNPS-cpam2015}), we only need to prove the following two properties.

\emph{Property 1.} For any $V\in C(\cA),$ the limit
\begin{eqnarray*}
  Q(V)=\lim_{k\rightarrow \infty} \frac{1}{k}\log \mE \exp\{\sum_{n=1}^kV(u_n)\}
\end{eqnarray*}
exists and does not depend on the initial condition $u_0.$

\emph{Property 2.} There is a dense vector space $\cV\subseteq C(\cA)$
such that there exists unique $\sigma_V \in  \PP(\cA)$ satisfying
\begin{eqnarray*}
  Q(V)=\langle V,\sigma_V\rangle-I(\sigma_V),
\end{eqnarray*}
where the $I$ is defined by (\ref{px-2}).

For any   $V\in \cV:=C^1(\cA),$   the generalized Markov kernal $P^V(u,\cdot)$ is given  by
\begin{eqnarray*}
  P^V(u,\Gamma)=\mE_u \big(I_{\Gamma}(u_1)e^{V(u_1)}\big), \quad u \in \mathcal{A},
  \Gamma\in \mathcal{B}(\mathcal{A}),
\end{eqnarray*} 
where $I_{\Gamma}(\cdot)$ is  the indicator function of  set $\Gamma.$
We define  $P_k^V(u,\Gamma)$ by the relations $P_0^V(u,\cdot)=\delta_u, P_1^V(u,\Gamma)=P^V(u,\Gamma)$ and
\begin{eqnarray*}
  P_k^V(u,\cdot):=\int_{\mathcal{A}}P_{k-1}^V(u,dv)P^V(v,\cdot),\quad  k\geq 2.
\end{eqnarray*}
 The  operators $\mathcal{B}_k^{V}: C(\cA)\rightarrow C(\cA)$
 and  $\mathcal{B}_k^{V*}:\mathcal{M}_{+}(\cA)\rightarrow \mathcal{M}_{+}(\cA)$
  are respectively defined by
 \begin{eqnarray*}
\mathcal{B}_k^{V}f(u)=\int_{\mathcal{A}} P_k^V(u,dv)f(v),
\quad
\mathcal{B}_k^{V*}\mu(\Gamma)=\int_{\mathcal{A}} P_k^V(u,\Gamma)\mu(du),
 \end{eqnarray*}
 where  $f\in C(\cA)$, $\mu\in \mathcal{M}_{+}(\cA)$, $\Gamma\in \BBBBB(\cA)$
 and $\mathcal{M}_{+}(\cA)$ is the space of non-negative Borel measures on $\cA$ endowed with
 the topology of weak convergence.
By (\ref{h-2}) and   the  following inequality
\begin{eqnarray*}
  P_k^V(u,\cdot)\geq e^{-k \|V\|_\infty}P_k(u,\cdot) \quad \text{ for any } u\in \cA,
\end{eqnarray*}
one finds   that  for some $C=C(k,V)>0,$
\begin{eqnarray*}
 C^{-1}\leq P_k^V(u,\cA)\leq C,\quad \text{ for all } u\in \cA.
\end{eqnarray*}
By this inequality and the  uniform Feller property, and using   \cite[Theorem 2.1]{JNPS-cpam2015} for $\mathcal{B}_k^V,$ $\mathcal{B}_k^{V*}$ and $P_k^V(u,\cdot)$,
  there are a number $\lambda_V>0,$ a strictly positive function $h_V\in C(\cA),$ and a measure $\mu_V\in \mathcal{P}(\cA)$ satisfying
\begin{eqnarray*}
  \mathcal{B}^V_1h_V=\lambda_Vh_V, \quad {\mathcal{B}_1^V}^*\mu_V=\lambda_V\mu_V,
  \quad \langle h_V, \mu_V \rangle =1,
\end{eqnarray*}
such that for any $f\in C(\cA)$ and $\nu \in \mathcal{M}_{+}(\cA)$ we have
\begin{eqnarray}
 \label{px-3} &&  \lambda_V^{-k}\mathcal{B}_k^Vf\rightarrow \langle f,\mu_V\rangle h_V \text{ in } C(\cA) \text{ as } k\rightarrow +\infty,
  \\ \label{px-4}
    && \lambda_V^{-k}{\mathcal{B}_k^V}^*\nu \rightharpoonup
    \langle h_V,\nu \rangle  \mu_V   \text{ in } \mathcal{M}_{+}(\cA) \text{ as } k\rightarrow +\infty.
\end{eqnarray}
Following the argument  in  \cite[Pages 2134--2135]{JNPS-cpam2015},
we can obtain the   Property  1 and Property 2  from  (\ref{px-3}) and (\ref{px-4}).
The proof is complete.
\end{proof}


\subsection{Proof of  Uniform Irreducibility }
\label{subsection3:1}
The proof of  uniform irreducibility is inspired by
the method in \cite{Kuksin_Shirikyan_2000_CMaPh}.
\begin{lemma}
\label{x-1}
  For any $\rho>0$ and integer $M\geq 1$ there is $p_0=p_0(\rho,M)>0$
  such that
  \begin{eqnarray*}
    \mP\big(\|\eta_j-x_j\|_{L^2(Q)}<\rho, 1\leq j\leq M\big)\geq p_0
  \end{eqnarray*}
uniformly for $x_1,\cdots,x_M\in  \KK.$
\end{lemma}
\begin{proof}
Recall that $\KK$ is a compact subset of $H_0^1(Q,\mR^2)$ and thus a compact subset of $L^2(Q)$,
there exist $n_0$ elements $y_1,...,y_{n_0} \in \KK$ such that $B_{H}(y_1, \rho/2),...,B_{H}(y_{n_0},\rho/2)$ cover $\KK$.
Since $\{\eta_j\}$ are i.i.d random variables with support in $\KK$, we know
 \begin{eqnarray*}
    \mP\big(\|\eta_j-x_j\|_{L^2(Q)}<\rho, 1\leq j\leq M\big)&=& \prod_{j=1}^M \mP\big(\|\eta_j-x_j\|_{L^2(Q)}<\rho\big)
  \end{eqnarray*}
 {As} $\KK$ is the support of $\eta_1$, we know $\tilde p_0:= \min_{1 \le j \le n_0} \mP\big(\|\eta_1-y_j\|_{L^2(Q)}<\rho/2\big)>0$. It is easy to see that for all $j$
$\mP\big(\|\eta_j-x_j\|_{L^2(Q)}<\rho\big) \ge \mP\big(\|\eta_j-y_{n(j)}\|_{L^2(Q)}<\rho/2\big) \ge \tilde p_0$ where $y_{n(j)}$ is the point in $\{y_1,...,y_{n_0}\}$ such that $x_j \in B_{H}(y_{n(j)},\rho/2)$. Taking $p_0=\tilde p_0^M$, we immediately conclude the proof.
\end{proof}
\begin{lemma}
\label{1-1}
  For any $r>0$ there is an integer  $k\geq 0$ such that $\cA$
  is contained in the $r$-neighbourhood of $\cA_k$, i.e.,
  for any $a\in \cA$ there exists $a_k\in \cA_k$ such that $a_k\in B_H(a,r).$
\end{lemma}
\begin{proof}
We claim that $\cA_{j+1} \supset \cA_j$ for all $j$. To see this, for any $u \in \cA_j$, we know $u \in \cA_{j+1}$ since $0 \in \KK$ and thus $u$ itself is a (constant) solution to the corresponding system \eqref{3-1}.

  Since $\cA$ is the closure of $\cup_{j=0}^\infty \cA_j$, for any $r>0$ we have
   \begin{eqnarray*}
     \cA\subset  \cup_{j=0}^\infty O_j,
   \end{eqnarray*}
   where $O_j$ is the open $r$-neighbourhood of $\cA_j$ in $H.$
   Since $\{\cA_j\}$ are an increasing sequence, so are $\{O_j\}$.  By the compactness of $\cA$, we complete the proof.

\end{proof}
\begin{lemma}
\label{x-2}
  For any  $r>0$ there are positive constants  $\eps, \delta$ and an integer $k>0$ such that
  \begin{eqnarray*}
    P_{k}(u,B_H(a,r))\geq \eps, \quad \forall u\in B_H(\hat{u},\delta) \text{ and } a\in \cA,
  \end{eqnarray*}
  where $\hat u$ is defined in {\bf (H2)} and $\{P_{k}\}$ is the transition probability family of the Markov chain $\{u_{k}\}$.
\end{lemma}
\begin{proof}
 By  Lemma \ref{1-1},   for any  $r>0$, there is  a   $k=k(r)\geq 1 $  such that  for  any $a\in \cA$, the following  
   \begin{eqnarray*}
   \|a-a_k\|\leq \frac{r}{2}
 \end{eqnarray*}
 holds for some   $a_k\in \cA_k$.
 Since $a_k\in \cA_k,$ there are  $\theta_j^0\in  \KK$ with $j=1,...,k$
 such that
 \begin{eqnarray*}
   S_k(\hat{u},\theta_1^0,\cdots,\theta_k^0 )=a_k,
 \end{eqnarray*}
 where $S_k$ is given in hypothesis (\textbf{H2}).
 Noticing that $S$ is a continuous mapping from  $ H\times L^2(Q)$ to $ H$,
 the following holds
 \begin{eqnarray*}
   S_k(u,\theta_1,\cdots,\theta_k)\in B_H(a_k,\frac{r}{2})
 \end{eqnarray*}
 if
 \begin{eqnarray*}
   \|u-\hat{u}\|<\delta, \quad  \|\theta_j-\theta_j^0\|_{L^2(Q)}<\delta, \quad j=1,\cdots,k,
 \end{eqnarray*}
 for sufficiently small $\delta$.
For any  $u\in B_H(\hat{u},\delta)$,   with the help of
 \begin{eqnarray*}
   P_k(u,B_H(a,r))&\geq& \mP(  S_k(u,\eta_1,\cdots,\eta_k)\in B_H(a_k,\frac{r}{2}))
   \\ &\geq& \mP( \|\eta_j-\theta_j^0\|_{L^2(Q)}<\delta, \quad j=1,\cdots,k)
 \end{eqnarray*}
 { and}  Lemma  \ref{x-1}, we complete our proof immediately.
\end{proof}
\begin{proposition}
\label{7-1}
  The Markov chain $\{u_k,k\geq 0\}$  associated with  (\ref{4-1})  has a unique stationary measure $\mu$.
  Moreover,  the followings  holds:
  \begin{itemize}
    \item[(\romannumeral1)] Uniform Irreducibility  on  $\cA:$  for any $r>0,$
    there is an integer $m\geq 1$ and a constant $p>0$ such that
    \begin{eqnarray*}
      P_m(x,B_H(a,r))\geq p,\quad \forall x,a \in \cA;
    \end{eqnarray*}
    \item[(\romannumeral2)]   $ \cA=\text{supp }\mu.$
  \end{itemize}
\end{proposition}
\begin{proof}
(i) By the approximate controllability hypothesis  \textbf{(H2)}
{ and the compactness of $\mathcal{A}$}, for any
$\delta>0,$ there exist  some $\eps_1>0$ and some integer  $L_1>0$ such that
\begin{eqnarray}
\label{p-1}
  P_{L_1}(x,B_H(\hat{u},\delta))\geq \eps_1,\quad \forall x\in \cA.
\end{eqnarray}
In view of    Lemma \ref{x-2},  for any  $r>0,$ there exist  positive constants  $\eps_2,\delta$ and  integer $L_2>0$ such that
\begin{eqnarray}
\label{p-2}
  P_{L_2}(u,B_H(a,r))\geq \eps_2, \quad \forall a \in \cA, \forall u\in B_H(\hat{u},\delta).
\end{eqnarray}
By (\ref{p-1}) and (\ref{p-2}),  we obtain
\begin{eqnarray}
\nonumber && P_{L_1+L_2}(x,B_H(a,r))=\int_{H}P_{L_1}(x,\dif u)P_{L_2}(u, B_H(a,r))
\\ \nonumber &&\geq P_{L_1}(x,B_{H}(\hat{u},\delta))\inf_{u\in B_{H}(\hat{u},\delta)}P_{L_2}(u, B_H(a,r))
\\ \label{1-9} &&\geq \eps_1\eps_2>0,
\end{eqnarray}
which implies   (\romannumeral1).


(ii) Since  $\cA$ is an invariant subset for (\ref{4-1}), it carries a stationary measure. Since the stationary measure is unique, we must have
 \begin{eqnarray*}
 \text{supp } \mu\subset \cA.
 \end{eqnarray*}
 On the other hand,
by (\ref{1-9}),  for any $a\in \cA,$  one sees that
\begin{eqnarray*}
 \mu( B_H(a,r) ) =\int_{\cA} \mu(\dif x)P_{L_1+L_2}(x,B_H(a,r))
 \geq \eps_1\eps_2>0,
\end{eqnarray*}
which implies  $ \cA\subset \text{supp }\mu.$
The prof of  (\romannumeral2) is complete.

\end{proof}

\subsection{Proof of uniform Feller property}
\label{subsection3:2}
Nersesyan et. al \cite{NPX-2022} proved uniform Feller property by Malliavin calculus for stochastic 2D Navier-Stokes equation driven by
highly degenerate noise, which cannot be applied in our setting because the noise of \eqref{3-1} is not Gaussian type. In contrast, our approach is via a coupling established by Shirikyan \cite{shirikyan-asens2015}, it takes a  role similar to  the Malliavin calculus in \cite{NPX-2022}.

\subsubsection{A coupling by Shirikyan \cite{shirikyan-asens2015}}
For any $\eps>0,$ define a symmetric function $d_\eps:H \times H\rightarrow \mR$ by the relation
\begin{eqnarray*}
  d_\eps (u_1,u_2)=
  \Big\{
  \begin{split}
    & 1, \quad \text{if } \|u_1-u_2\|>\eps,
    \\ & 0, \quad \text{if } \|u_1-u_2\|\leq \eps
  \end{split}
\end{eqnarray*}
and introduce the following function on the space $\mathcal{P}(H)\times \mathcal{P}(H)$:
\begin{eqnarray*}
  K_\eps(\mu_1,\mu_2)=\sup_{f,g}(\langle f,\mu_1\rangle-\langle g,\mu_2\rangle),
\end{eqnarray*}
where the supremum is taken over all functions $f,g\in C(H)$ satisfying
\begin{eqnarray*}
  f(u_1)-g(u_2)\leq d_\eps(u_1,u_2),\forall u_1,u_2\in H.
\end{eqnarray*}

We have  the following proposition, whose proof is very similar to the argument in \cite{shirikyan-asens2015}, but we give its details for the completeness.
\begin{proposition}
\label{1-3}
  For any $q\in (0,1),$ there exist positive  constants  $d=d(q)>0$ and $C=C(d,q)$
  such that for any  two points $u_0,u_0'\in \cA$ satisfying the inequality $\|u_0-u_0'\|\leq d$
  the pair $(P_1(u_0,\cdot),P_1(u_0',\cdot))$ admits a coupling $(V(u_0,u_0'),V'(u_0,u_0'))$ such that
  \begin{eqnarray}
  \label{j-4}
    \mP\big\{ \|V(u_0,u_0')-V'(u_0,u_0')\|>q\|u_0-u_0'\|\big\}\leq C\|u_0-u_0'\|
  \end{eqnarray}
  and  the functions  $V,V': \Omega \times Z\rightarrow H$ are measurable, where
  \begin{eqnarray}
  \label{j-6}
    Z=\{(u_0,u_0')\in \cA \times \cA: \|u_0-u_0'\|\leq d\}.
  \end{eqnarray}
\end{proposition}
\begin{proof}

We  fix $\tilde{R}>0$ so  large that  $\cA\subseteq B_H(\tilde{R}-1)$ and $\|\eta_k\|_{H^1(Q)}\leq \tilde{R}-1$
 with probability 1.

By  \cite[Theorem 3.1]{shirikyan-asens2015}, for  $q\in (0,1),$
there exist  positive constants $C,d$,  integer $m$  and  a continuous   mapping
$$\Phi:B_{\tilde{R}}   \times B_H(\tilde{R})\rightarrow \cL(H,\mathcal{E}_m)$$
\footnote{$\cL(H,\mathcal{E}_m)$ denotes the space of all linear operators from $H$ to
$\mathcal{E}_m$.  The definitions of $B_{\tilde{R}}, B_H(\tilde{R})$ and $\mathcal{E}_m$ are given in Page 4.
}such that  the following properties hold:

\noindent \textbf{Contraction.}
For any $h \in B_{\tilde{R}}$ and $u_0,u_0'\in B_H(\tilde{R})$ with $\|u_0-u_0'\|\leq d$, we have
\begin{eqnarray}
\label{h-1}
  \|S(u_0,h)-S(u_0',h+\Phi(h,u_0)(u_0'-u_0))\|\leq \frac{q\|u_0-u_0'\|}{2}.
\end{eqnarray}

\noindent \textbf{Regularity.}
The mapping $\Phi$ is  infinitely smooth in the Fr\'echet sense.

\noindent \textbf{Lipschitz continuity.}
The mapping $\Phi$ is Lipchitz continuous with the constant $C$, i.e.,
\begin{eqnarray*}
  \|\Phi(h_1,\hat{u}_1)-\Phi(h_2,\hat{u}_2)\|_{\cL}\leq C(\|h_1-h_2\|_{H^1(Q)}+\|\hat{u}_1-\hat{u}_2\|),
\end{eqnarray*}
where  $\|\cdot\|_{\cL} $ stands for the norm in the space $\cL(H,\mathcal{E}_m).$

By \cite[Proposition 5.3]{shirikyan-asens2015},   there exists a coupling
$(V(u_0,u_0'),V'(u_0,u_0'))$ such that
\begin{eqnarray}
\label{j-3}
\mP\big\{ \|V(u_0,u_0')-V'(u_0,u_0')\|>\eps \big\}\leq K_{\eps/2}( P_1(u_0,\cdot),P_1(u_0',\cdot))
\end{eqnarray}
holds with $\eps=q\|u_0-u_0'\|$  and  the functions  $V,V': \Omega \times Z \rightarrow H$ are measurable.

It suffices to bound  $K_{\eps/2}( P_1(u_0,\cdot),P_1(u_0',\cdot))$.
Define the transformation $\Psi=\Psi_{u_0,u_0'}$  of the space  $H^1(Q)$  by the following relation
\begin{eqnarray}
\label{j-5}
\Psi(h)=h+\varrho(\|h\|_{H^1(Q)})\Phi(h,u_0)(u_0'-u_0)
\end{eqnarray}
where  $\varrho$ is a smooth function such that $\varrho(a)=1$ for $a\leq \tilde{R}-1$ and
$\varrho(a)=0$ for $a\geq \tilde{R}.$
We denote the law of $\eta_k$ and $\Psi(\eta_k)$  on $B_{\tilde{R}}$ by $\lambda$ and
$\Psi_*(\lambda)$, respectively.
With the help  of    (\ref{h-1})  and \cite[Proposition 5.2]{shirikyan-asens2015},
one has
\begin{eqnarray}
\label{j-1}
K_{\eps/2}( P_1(u_0,\cdot),P_1(u_0',\cdot))\leq 2 \|\lambda-\Psi_*(\lambda)\|.
\end{eqnarray}
With the help of  the   Lipschitz continuity of functions  $\Phi$,
$\Psi$,  by  the definition  (\ref{j-5}), we can verify the conditions of   \cite[Proposition 5.6]{shirikyan-asens2015} with a constant $\kappa$ proportional to $\|u_0-u_0'\|$.
Therefore,
\begin{eqnarray}
  \label{j-2}
   \|\lambda-\Psi_*(\lambda)\| \leq C\|u_0-u_0'\|.
\end{eqnarray}
Combining  (\ref{j-1}), (\ref{j-2}) and (\ref{j-3}), we obtain the desired result  (\ref{j-4}).
\end{proof}

\subsubsection{Proof of uniform Feller property}
\begin{theorem}
\label{2-3}
Let
 Hypotheses  \textbf{(H1)},\textbf{(H2)} and  \textbf{(H3)} hold,
 then for  any $V,f\in C^1(\cA)$,   the family $\{\|\mathcal{B}_n^V\textbf{1}\|_{\infty}^{-1}\mathcal{B}_n^Vf,n\geq 0 \}$
is equicontinuous  on $\cA$. Furthermore,  there is a   $\gamma>0$  such that
\begin{eqnarray} \label{e:UFP}
\begin{split}
& \left|\mathcal{B}_n^Vf(u_0)-\mathcal{B}_n^Vf(u_0') \right|
\\ & \leq C \|\mathcal{B}_n^V\textbf{1}\|_{\infty}(\|f\|_\infty+e^{-\gamma n}\|\nabla f\|_\infty)\|u_0-u_0'\|,
\end{split}
\end{eqnarray}
holds for any $u_0,u_0'\in \cA$ and $n\geq 1$.  Here $C=C(\gamma,\|V\|_\infty,\|\nabla V\|_\infty).$ In particular,
uniform Feller property holds.
\end{theorem}
\begin{proof}
Uniform Feller property follows from \eqref{e:UFP} immediately. Let us prove \eqref{e:UFP} in the following four steps.

 \emph{
Step 1: construction of coupling processes.}
This coupling is borrowed from \cite{shirikyan-asens2015}, here we give the details for the completeness.
Let $q=q(\|V\|_\infty)\in (0,1)$  be  a constant that will be  determined  later.
By  Proposition  \ref{1-3}, there exist positive  constants  $d=d(q)>0$ and $C=C(d,q)$
  such that for any  two points $u_0,u_0'\in \cA$ with $\|u_0-u_0'\|\leq d$,
  the pair $(P_1(u_0,\cdot),P_1(u_0',\cdot))$ admits a coupling $(V(u_0,u_0'),V'(u_0,u_0'))$ such that (\ref{j-4})
   holds and that the functions  $V,V': \Omega \times Z\rightarrow H$ are measurable, here
    $Z$ is given by (\ref{j-6}).
We now define a coupling operator  by the relation
\begin{eqnarray}
\label{k-1}
  \cR(u_0,u_0')=
  \left\{
  \begin{split}
  & (V(u_0,u_0'),V'(u_0,u_0')), & \text{for } \|u_0-u_0'\|\leq d,
  \\
  &  (S(u_0,\zeta),S(u_0',\zeta') ), & \text{for } \|u_0-u_0'\|>d,
  \end{split}
  \right.
\end{eqnarray}
where  $\zeta$ and $\zeta'$ are independent random variables whose law coincides  with that of  $\eta_1.$ Without loss of generality, we assume that $\zeta,\zeta',V$ and $V'$ are all defined on the same probability space. To stress   the dependence on $\omega,$ we shall sometimes write $\cR(u_0,u_0';\omega)$ instead of  $\cR(u_0,u_0').$

Let $(\Omega_k,\cF_k,\mP_k), k
\geq 1$   be countable many copies of the
probability space on which $\cR$
 is defined and let $(\Omega,\cF,\mP)$ be the direct product of these spaces.
 For   $u_0,u_0'\in \cA,$ set
 \begin{eqnarray*}
   (u_k,u_k')=\cR(u_{k-1},u_{k-1}'; \omega_k), ~~k\geq 1.
 \end{eqnarray*}
For the filtration generated by $\{(u_i,u_i'), 0\leq i\leq k\}$, we denote it by
$\mathscr{F}_k.$
We define the following events which will be used below:
\begin{align*}
  A_0   &=  \{\omega:\|u_1-u_1'\|>q \|u_0-u_0'\|\},
 \\  B_i &=  \{\omega:\|u_i-u_i'\|\leq q \|u_{i-1}-u_{i-1}'\|\},
  \\  A_k   &=  \big(\cap_{i=1}^kB_i\big)\cap B_{k+1}^c, \quad \forall 1\leq k\leq n-1,
  \\ A_{n}   &= \cap_{i=1}^{n}B_i.
\end{align*}


\emph{
Step 2: stratification.}
 We write
\begin{eqnarray} \label{2-10}
\nonumber && \mathcal{B}_n^Vf(u_0)-\mathcal{B}_n^Vf(u_0')=\sum_{k=0}^{n} J_k,
\end{eqnarray}
where
\begin{align*}
  J_{k}   & = \mE \big\{I_{A_k}
  \big[\exp\{\sum_{i=1}^nV(u_i)\}f(u_n)-
  \exp\{\sum_{i=1}^nV(u_i')\}f(u_n')\big]
  \big\}, ~0\leq k\leq n.
\end{align*}
In the below, we will use the notation $\sum_{i=m_1}^{m_2}c_i=0$
 if  $m_1>m_2.$

\emph{
Step 3: estimates of  $J_k$, $0\leq k\leq n-1$.} By triangle inequality,
$$|J_{k}| \le |J_{k,1}|+|J_{k,2}|,$$
with $J_{k,1}=\mE \big\{I_{A_k}
  \exp\{\sum_{i=1}^nV(u_i)\}f(u_n)
  \big\}$ and $J_{k,2}=\mE \big\{I_{A_k}
  \exp\{\sum_{i=1}^nV(u_i')\}f(u_n')
  \big\}$
Noticing that  $\cA$ is an invariant subset for (\ref{4-1}), we have
\begin{eqnarray}
  \nonumber |J_{k,1}|&=&
  \left|\mE \Big\{I_{A_k}
  \exp\{\sum_{i=1}^{k+1}V(u_i)\}\mE \big[\exp\{\sum_{i=k+2}^{n}V(u_i)\}f(u_n)|\mathscr{F}_{k+1}\big]
  \Big\}\right|
  \\
  \label{1-4}  &\leq & \|f\|_\infty  e^{\|V\|_\infty(k+1)}
  \|\mathcal{B}_n^V\textbf{1}\|_{\infty}\mP(A_k).
\end{eqnarray}
By Proposition \ref{1-3}, for $1 \le k \le n-1$,
\begin{eqnarray*}
 \mP(A_k)&=&\mP\big((\cap_{i=1}^kB_i)\cap B_{k+1}^c\big) \\
&=&\mE \Big(I_{\cap_{i=1}^kB_i}\mP\big( B_{k+1}^c \big| \mathscr{F}_k \big)\Big) \\
&\leq&  C_q   \mE \Big(I_{\cap_{i=1}^kB_i} \|u_k-u_k'\|\Big)  \\
&\leq& C_q  q^k\|u_0-u_0'\|,
\end{eqnarray*}
here  $C_q$ is a constant independent of $k.$
Obvious, by Proposition \ref{1-3}, the above estimate also holds for $k=0.$
Combining this inequality with  (\ref{1-4}), we get
\begin{eqnarray}
\label{xin1-6}
  \left|J_{k,1}\right| \leq C_q \|f\|_\infty \|\mathcal{B}_n^V\textbf{1}\|_{\infty}
  ~e^{\|V\|_\infty(k+1)}
  q^k  \|u_0-u_0'\|.
\end{eqnarray}
It is obvious that the same inequality holds for $J_{k,2}$. Hence,
$$|J_{k}| \le C_q \|f\|_\infty \|\mathcal{B}_n^V\textbf{1}\|_{\infty}
  ~e^{\|V\|_\infty(k+1)}
  q^k  \|u_0-u_0'\|.$$

\emph{
Step 4: estimate of $J_n$.}
 Using the fact
\begin{eqnarray*}
  \prod_{k=1}^{\ell }a_k-\prod_{k=1}^{\ell }b_k=\sum_{k=1}^{\ell }a_1\cdots a_{k-1}(a_k-b_k)b_{k+1}\cdots b_{\ell}
\end{eqnarray*}
with  $\ell=n+1,a_{n+1}=f(u_n),b_{n+1}=f(u_n')$
 and
\begin{eqnarray*}
  a_k=\exp\{V(u_k)\},~ b_k=\exp\{V(u_k')\}, \quad 1\leq k\leq n,
\end{eqnarray*}
we arrive at
\begin{eqnarray*}
   && J_n =\sum_{k=1}^n   \mE \Big\{I_{A_n}
  \big[\exp\{\sum_{i=1}^{k-1}V(u_i)\}\big(\exp\{V(u_k)\}-\exp\{V(u_k')\} \big)
  \\ && \quad \quad  \times
  \exp\{\sum_{i=k+1}^nV(u_i')\}f(u_n')\big]
  \Big\}
 +\mE \big\{I_{A_n}\exp\{\sum_{i=1}^nV(u_i)\}\big(f(u_n)-f(u_n')\big)\big\}
  \\ &&:=\sum_{k=1}^{n+1} J_{n,k}.
\end{eqnarray*}
First, consider  the terms   $J_{n,k}, 1\leq k\leq n.$
Notice that, on the event $A_n$, one has
\begin{eqnarray}
\label{h-3}
\|u_k-u_k'\|\leq q^k\|u_0-u_0'\|.
\end{eqnarray}
Therefore,
\begin{eqnarray*}
J_{n,k} &\leq &e^{k\|V\|_\infty}\mE \big\{I_{A_n}
  \|u_k-u_k'\|
  \exp\{\sum_{i=k+1}^nV(u_i')\}f(u_n')\big]
  \big\}
  \\
  &\leq  &e^{k\|V\|_\infty}\mE \Big\{ q^k \|u_0-u_0'\|~\mE\big\{
  \exp\{\sum_{i=k+1}^nV(u_i')\}f(u_n')\big|\mathscr{F}_{k} \big\}
  \Big\}
  \\ &\leq & \|f\|_\infty  e^{k\|V\|_\infty} \|\mathcal{B}_n^V\textbf{1}\|_{\infty}~ q^k \|u_0-u_0'\|.
\end{eqnarray*}
Next, consider the term $J_{n,n+1}.$ By a similar argument, we have
\begin{eqnarray*}
 J_{n,n+1} &\leq &e^{n\|V\|_\infty} \|\nabla f\|_\infty q^n \|u_0-u_0'\|.
\end{eqnarray*}
Combining the above estimates of $J_{n,k}, 1\leq k\leq n+1,$ we obtain
\begin{align}
   \nonumber   J_n & \leq \sum_{k=1}^{n}\|f\|_\infty  e^{k\|V\|_\infty} \|\mathcal{B}_n^V\textbf{1}\|_{\infty}~q^k \|u_0-u_0'\|
  \\    \label{xin1-5} &\quad +e^{n\|V\|_\infty} \|\nabla f\|_\infty q^n \|u_0-u_0'\|.
\end{align}

By  setting   $q$ small enough such that $q e^{\|V\|_\infty}\leq e^{-\gamma}$,
by  (\ref{xin1-5}) and (\ref{xin1-6}), we complete our proof.
\end{proof}


\def\cprime{$'$} \def\cprime{$'$}
  \def\polhk#1{\setbox0=\hbox{#1}{\ooalign{\hidewidth
  \lower1.5ex\hbox{`}\hidewidth\crcr\unhbox0}}}
  \def\polhk#1{\setbox0=\hbox{#1}{\ooalign{\hidewidth
  \lower1.5ex\hbox{`}\hidewidth\crcr\unhbox0}}}
  \def\polhk#1{\setbox0=\hbox{#1}{\ooalign{\hidewidth
  \lower1.5ex\hbox{`}\hidewidth\crcr\unhbox0}}} \def\cprime{$'$}
  \def\polhk#1{\setbox0=\hbox{#1}{\ooalign{\hidewidth
  \lower1.5ex\hbox{`}\hidewidth\crcr\unhbox0}}} \def\cprime{$'$}
  \def\cprime{$'$} \def\cprime{$'$} \def\cprime{$'$}

\end{document}